\documentclass[a4paper,11pt]{amsart}
\usepackage{geometry}
\usepackage{fullpage}
\usepackage[parfill]{parskip}   
\usepackage{graphicx}
\usepackage{amssymb}
\usepackage{amsmath}
\usepackage{epstopdf}
\usepackage{amsmath}

\usepackage{amsmath,amsthm,amscd,amssymb}
\usepackage{latexsym}
\usepackage[colorlinks,citecolor=red,pagebackref,hypertexnames=false]{hyperref}
 
\numberwithin{equation}{section}

\theoremstyle{plain}
\newtheorem{theorem}{Theorem}[section]
\newtheorem{Def}[theorem]{Definition}
\newtheorem{lemma}[theorem]{Lemma}

\theoremstyle{definition}

\theoremstyle{remark}

\newtheorem{case[theorem]}{Case}

\def \R{{\mathbb R}}
\def \N{{\mathbb N}}

\def \C{{\mathbb C}}

\def\H{{\mathbb H}}

\def\norm#1.#2.{\lVert#1\rVert_{#2}}

\def\R{\mathbb R}

\def \S{{\mathcal S}}

\def \H{{\mathcal H}}

\title{Uncertainty principles for the Opdam--Cherednik transform on modulation spaces}

\author{Anirudha Poria}

\address{Department of Mathematics,
Indian Institute of Science,
Bengaluru 560012, Karnataka, India.}

\email{anirudhap@iisc.ac.in}


\date{\today}

\begin{document}
\maketitle
\begin{abstract} 
In this paper, we establish the Cowling--Price's, Hardy's and Morgan's uncertainty principles for the Opdam--Cherednik transform on modulation spaces associated with this transform. The proofs of the theorems are based on the properties of the heat kernel associated with the Jacobi--Cherednik operator and the versions of the Phragm{\'e}n--Lindl{\"o}f type result for the modulation spaces.
\end{abstract}    

\section{Introduction} 

The uncertainty principle states that a non-zero function and its Fourier transform cannot be simultaneously sharply localized. There are various ways of measuring localization of a function and depending on it one can formulate different forms of the uncertainty principle. Uncertainty principles can be subdivided into quantitative and qualitative uncertainty principles. Quantitative uncertainty principles are some special inequalities which give us information about how a function and its Fourier transform relate. For example, Benedicks \cite{ben85}, Donoho and Stark \cite{don89}, and Slepian and Pollak \cite{sle61} gave qualitative uncertainty principles for the Fourier transforms. Qualitative uncertainty principles imply the vanishing of a function under some strong conditions on the function.  For example,  Hardy \cite{har33}, Morgan  \cite{mor34}, Cowling and Price \cite{cow83}, and Beurling \cite{hor91} theorems are the qualitative uncertainty principles. More precisely, Hardy \cite{har33} obtained the following uncertainty principle concerning the decay of a measurable function and its Fourier transform at infinity. 
\begin{theorem}[Hardy]\label{Hardy}
Let $f$ be a measurable function on $\R$ such that
\[ |f(x)| \leq C e^{-a  x^2}  \qquad \mathrm{and} \qquad  |\hat{f}(\xi)| \leq C e^{-b  \xi^2} \]
for some constants $a, b, C > 0$. Then three cases can occur.
\begin{enumerate}
\item[(i)] If $ab> \frac{1}{4}$, then $f = 0$ almost everywhere.
\item[(ii)] If $ab=\frac{1}{4}$, then the function $f$ is of the form $f(x)=C e^{-a  x^2}$, for some constant $C$.
\item[(iii)] If $ab<\frac{1}{4}$, then any finite linear combination of Hermite functions satisfies these decay conditions.
\end{enumerate} 
\end{theorem}
Cowling and Price \cite{cow83} generalized this theorem by replacing pointwise Gaussian bounds for $f$ by Gaussian bounds in $L^p$ sense and in $L^q$ sense for $\hat{f}$ as well. More precisely, they proved the following theorem.
\begin{theorem}[Cowling--Price]\label{Cow-Pri}
Let $f$ be a measurable function on $\R$ such that
\begin{enumerate}
\item[(i)] $\Vert e^{a x^2} f  \Vert_p < \infty,$
\item[(ii)] $\Vert e^{b \xi^2} \hat{f}  \Vert_q < \infty,$
\end{enumerate}
where $a, b>0$ and $1 \leq p, q \leq \infty$ such that $\min(p, q)$ is finite. If $ab \geq \frac{1}{4}$, then $f=0$ almost everywhere. If $ab < \frac{1}{4}$, then there exist infinitely many linearly independent functions satisfying $\mathrm{(i)}$ and $\mathrm{(ii)}$.
\end{theorem}
Over the years, analogues of Hardy's theorem have been extended to different settings (see \cite{tha04}). By replacing the Gaussian function $e^{a x^2}$ in Hardy's theorem by the function $e^{a |x|^\alpha}$ where $\alpha > 2$, Morgan \cite{mor34} obtained the following uncertainty principle. 
\begin{theorem}[Morgan]\label{mor}
Let $a>0$, $b>0$, and let $\alpha, \beta$ be positive real numbers satisfying $\alpha >2$ and $1/\alpha+1/\beta=1$. Suppose that $f$ is a measurable function on $\R$ such that  
\[ e^{a|x|^\alpha} f \in L^\infty(\R) \quad \text{and} \quad  e^{b|\lambda|^\beta} \hat{f} \in L^\infty(\R). \]
If $ (a \alpha)^{1/\alpha} (b \beta)^{1/\beta} > \left( \sin \left( \frac{\pi}{2}(\beta -1 ) \right) \right)^{1/\beta}, $ then $f = 0$ almost everywhere. 
\end{theorem}
A generalization of this theorem was obtained by Ben Farah and Mokni \cite{far03} where they proved an $L^p$ -- $L^q$-version of Morgan's theorem. For a more detailed study of uncertainty principles, we refer to the book of Havin and J\"oricke \cite{hav94}.

Considerable attention has been devoted to discovering generalizations to new contexts for the Cowling--Price's, Hardy's and Morgan's uncertainty principles. For instance, these theorems were obtained in \cite{mej16} for the generalized Fourier transform and in \cite{ray04} for symmetric spaces. Also, an $L^p$ version of Hardy's theorem was obtained for the Dunkl transform in \cite{gal04} and for motion groups in \cite{egu00}. As a generalization of Euclidean uncertainty principles for the Fourier transform, Daher et al. \cite{dah12} have obtained some uncertainty principles for the Cherednik transform. These theorems are further extended to the Opdam--Cherednik transform in \cite{mej014} using classical uncertainty principles for the Fourier transform and composition properties of the Opdam--Cherednik transform. However, upto our knowledge, these types of uncertainty principles have not been studied in the case of the modulation spaces. In this paper, we attempt to prove the Cowling--Price's, Hardy's and Morgan's uncertainty principles for the Opdam--Cherednik transform on modulation spaces associated with this transform.

The motivation to prove these uncertainty principles for the Opdam--Cherednik transform on modulation spaces arises from the classical uncertainty principles for the Fourier transform on the Lebesgue spaces. Since the last decade modulation spaces have found to be very fruitful in various current trends (e.g., pseudo-differential operators,  partial differential equations, etc..) of investigation and have been widely used in several fields in analysis, physics and engineering.  Uncertainty principles have implications in two main areas: quantum mechanics and signal analysis, and modulation spaces are widely used in these areas. We hope that the study of uncertainty principles for the modulation spaces makes a significant impact in these areas. Another important motivation to study the Jacobi--Cherednik operators arises from their relevance in the algebraic description of exactly solvable quantum many-body systems of Calogero--Moser--Sutherland type (see \cite{die00, hik96}) and they provide a useful tool in the study of special functions with root systems (see \cite{dun92, hec91}). These describe algebraically integrable systems in one dimension and have gained considerable interest in mathematical physics. Other motivation for the investigation of the Jacobi--Cherednik operator and the Opdam--Cherednik transform is to generalize the previous subjects which are bound with the physics. For a more detailed discussion, we refer to \cite{mej16}.

Since modulation spaces are much larger spaces than the Lebesgue spaces, can we determine the functions $f$ satisfying the conditions of Theorems \ref{Hardy}, or \ref{Cow-Pri}, or  \ref{mor} for the modulation spaces? In this paper, we answer these questions. The common key to obtaining extensions of uncertainty principles for the Opdam--Cherednik transform is a slice formula, that is, this transform is decomposed as a composition of the classical Fourier transform and the Jacobi--Cherednik intertwining operator (see \cite{mej014}). However, without using a slice formula for the Opdam--Cherednik transform we give the analogue of the uncertainty principles within the framework of the Opdam--Cherednik transform by using an estimate of the heat kernel, which obtained in \cite{fit89}. Here, we consider the modulation spaces associated with the Opdam--Cherednik transform, as the standard modulation spaces are not suited to this transform. In this paper, we prove the uncertainty principles by using the properties of the heat kernel associated with the Jacobi--Cherednik operator and the versions of the Phragm{\'e}n--Lindl{\"o}f type result for the modulation spaces.
The two lemmas,  Lemma \ref{lem4} for Cowling--Price's theorem and Lemma \ref{lem2} for Morgan's theorem are essential. Since these lemmas hold for the modulation spaces associated with the Opdam--Cherednik transform, which also satisfies H\"older's inequality, we can apply the classical process to obtain uncertainty principles for the Opdam--Cherednik transform on modulation spaces.

The paper is organized as follows. In Section \ref{sec2}, we recall some basic facts about the Jacobi--Cherednik operator and we discuss the main results for the Opdam--Cherednik transform. We also give some properties of the heat kernel associated with the Jacobi--Cherednik operator. In Section \ref{sec3}, we discuss the modulation spaces associated with the Opdam--Cherednik transform. In Section \ref{sec4}, we give the Phragm{\'e}n--Lindl{\"o}f type result for the modulation spaces and using it we prove an $M^p$ -- $M^q$-version of Cowling--Price's theorem for the Opdam--Cherednik transform. In Section \ref{sec5}, an analogue of the classical Hardy's theorem is obtained for the Opdam--Cherednik transform on modulation spaces associated with this transform. Finally, in Section \ref{sec6}, we obtain another version of the Phragm{\'e}n--Lindl{\"o}f type result for the modulation spaces and we prove an $M^p$ -- $M^q$-version of Morgan's theorem for the Opdam--Cherednik transform.

\section{Harmonic analysis and the Opdam--Cherednik transform}\label{sec2}

In this section, we collect the necessary definitions and results from the harmonic analysis related to the Opdam--Cherednik transform. The main references for this section are \cite{and15, mej14, opd95, opd00, sch08}. However, we will use the same notation as in \cite{joh15}.

Let $T_{\alpha, \beta}$ denote the Jacobi--Cherednik differential--difference operator (also called the Dunkl--Cherednik operator)
\[T_{\alpha, \beta} f(x)=\frac{d}{dx} f(x)+ \Big[ 
(2\alpha + 1) \coth x + (2\beta + 1) \tanh x \Big] \frac{f(x)-f(-x)}{2} - \rho f(-x), \]
where $\alpha, \beta$ are two parameters satisfying $\alpha \geq \beta \geq -\frac{1}{2}$ and $\alpha > -\frac{1}{2}$, and $\rho= \alpha + \beta + 1$. Let $\lambda \in \C$. The Opdam hypergeometric functions $G^{\alpha, \beta}_\lambda$ on $\R$ are eigenfunctions $T_{\alpha, \beta} G^{\alpha, \beta}_\lambda(x)=i \lambda  G^{\alpha, \beta}_\lambda(x)$ of $T_{\alpha, \beta}$ that are normalized such that $G^{\alpha, \beta}_\lambda(0)=1$. The eigenfunction $G^{\alpha, \beta}_\lambda$ is given by
\[G^{\alpha, \beta}_\lambda (x)= \varphi^{\alpha, \beta}_\lambda (x) - \frac{1}{\rho - i \lambda} \frac{d}{dx}\varphi^{\alpha, \beta}_\lambda (x)=\varphi^{\alpha, \beta}_\lambda (x)+ \frac{\rho+i \lambda}{4(\alpha+1)} \sinh 2x \; \varphi^{\alpha+1, \beta+1}_\lambda (x),  \]
where $\varphi^{\alpha, \beta}_\lambda (x)={}_2F_1 \left(\frac{\rho+i \lambda}{2}, \frac{\rho-i \lambda}{2} ; \alpha+1; -\sinh^2 x \right) $ is the classical Jacobi function.

For every $ \lambda \in \C$ and $x \in  \R$, the eigenfunction
$G^{\alpha, \beta}_\lambda$ satisfy
\[ |G^{\alpha, \beta}_\lambda(x)| \leq C \; e^{-\rho |x|} e^{|\text{Im} (\lambda)| |x|},\] 
where $C$ is a positive constant. Since $\rho > 0$, we have
\begin{equation}\label{eq1}
|G^{\alpha, \beta}_\lambda(x)| \leq C \; e^{|\text{Im} (\lambda)| |x|}. 
\end{equation}
Let us denote by $C_c (\R)$ the space of continuous functions on $\R$ with compact support.
\begin{Def}
Let $\alpha \geq \beta \geq -\frac{1}{2}$ with $\alpha > -\frac{1}{2}$. The Opdam--Cherednik transform $\mathcal{H} f$ of a function $f \in C_c(\R)$ is defined by
\[ \H f (\lambda)=\int_{\R} f(x)\; G^{\alpha, \beta}_\lambda(-x)\; A_{\alpha, \beta} (x) dx \quad \text{for all } \lambda \in \C, \] 
where $A_{\alpha, \beta} (x)= (\sinh |x| )^{2 \alpha+1} (\cosh |x| )^{2 \beta+1}$. The inverse Opdam--Cherednik transform for a suitable function $g$ on $\R$ is given by
\[ \H^{-1} g(x)= \int_{\R} g(\lambda)\; G^{\alpha, \beta}_\lambda(x)\; d\sigma_{\alpha, \beta}(\lambda) \quad \text{for all } x \in \R, \]
where $$d\sigma_{\alpha, \beta}(\lambda)= \left(1- \dfrac{\rho}{i \lambda} \right) \dfrac{d \lambda}{8 \pi |C_{\alpha, \beta}(\lambda)|^2}$$ and 
$$C_{\alpha, \beta}(\lambda)= \dfrac{2^{\rho - i \lambda} \Gamma(\alpha+1) \Gamma(i \lambda)}{\Gamma \left(\frac{\rho + i \lambda}{2}\right)\; \Gamma\left(\frac{\alpha - \beta+1+i \lambda}{2}\right)}, \quad \lambda \in \C \setminus i \mathbb{N}.$$
\end{Def}

The Plancherel formula is given by 
\begin{equation}\label{eq03}
\int_{\R} |f(x)|^2 A_{\alpha, \beta}(x) dx=\int_\R \H f(\lambda) \overline{\H \check{f}(-\lambda)} \; d \sigma_{\alpha, \beta} (\lambda),
\end{equation}
where $\check{f}(x):=f(-x)$.

Let $L^p(\R,A_{\alpha, \beta} )$ (resp. $L^p(\R, \sigma_{\alpha, \beta} )$), $p \in [1, \infty] $, denote the $L^p$-spaces corresponding to the measure $A_{\alpha, \beta}(x) dx$ (resp. $d | \sigma_{\alpha, \beta} |(x)$). The Schwartz space $\S_{\alpha, \beta}(\R )=(\cosh x )^{-\rho} \S(\R)$ is defined as the space of all differentiable functions $f$ such that 
$$ \sup_{x \in \R} \; (1+|x|)^m e^{\rho |x|} \left|\frac{d^n}{dx^n} f(x) \right|<\infty,$$ 
for all $m, n \in \N_0 = \N \cup \{0\}$, equipped with the obvious seminorms. The Opdam--Cherednik transform $\H$ and its inverse $\H^{-1}$ are topological isomorphisms between the space $\S_{\alpha, \beta}(\R )$ and the space $\S(\R)$ (see \cite{sch08}, Theorem 4.1).

Let $t > 0$. The heat kernel $E^{\alpha, \beta}_t$ associated with the Jacobi--Cherednik operator is defined by
\begin{equation}\label{eq04}
E^{\alpha, \beta}_t(x)=\H^{-1}(e^{-t \lambda^2})(x) \quad \text{for all } x \in \R.
\end{equation}
For all $t > 0$, $E^{\alpha, \beta}_t$ is an $C^\infty$-function on $\R$. Moreover, for all $t > 0$ and all $\lambda \in \R$, we have
\begin{equation}\label{eq05}
\H (E^{\alpha, \beta}_t) (\lambda)=e^{-t \lambda^2}.
\end{equation}
We refer to \cite{cho03} for further properties of the heat kernel $E^{\alpha, \beta}_t$. From (\cite{fit89}, Theorem 3.1), there exist two real numbers $\mu_1$ and $\mu_2$, such that
\begin{equation}\label{eq06}
\frac{e^{\mu_1 t}}{2^{2\alpha+1} \Gamma(\alpha+1) t^{\alpha+1}} \frac{e^{-\frac{x^2}{4t}}}{\sqrt{B_{\alpha, \beta}(x)}} \leq E^{\alpha, \beta}_t(x) \leq  \frac{e^{\mu_2 t}}{2^{2\alpha+1} \Gamma(\alpha+1) t^{\alpha+1}} \frac{e^{-\frac{x^2}{4t}}}{\sqrt{B_{\alpha, \beta}(x)}}, \quad \forall x \in \R,
\end{equation}
where $B_{\alpha, \beta} (x)= (\sinh |x|/|x| )^{2 \alpha+1} (\cosh |x| )^{2 \beta+1}$ for all $x \in \R \setminus \{0\}$ and $B_{\alpha, \beta}(0)=1$. Also, we have $A_{\alpha, \beta}(x)= |x|^{2\alpha+1} B_{\alpha, \beta}(x)$ and for all $x \in \R$, $B_{\alpha, \beta} (x) \geq 1$.

\section{Modulation spaces associated with the Opdam--Cherednik transform}\label{sec3}
The modulation spaces were introduced by Feichtinger \cite{fei03, fei97}, by imposing integrability conditions on the short-time Fourier transform (STFT) of tempered distributions. More specifically, for $x, w \in \R$, let $M_w$ and $T_x$ denote the operators of modulation and translation. Then, the STFT of a function $f$ with respect to a window function $g \in  \S(\R)$ is defined by
\[ V_g f (x,w)=\langle f, M_w T_x g \rangle=\int_{\R} f(t) \overline{g(t-x)} e^{-2 \pi i w t} dt, \quad (x,y) \in \R^2. \] 

Here we are interested in modulation spaces with respect to measure $A_{\alpha, \beta}(x) dx$.  
\begin{Def}
Fix a non-zero window $g \in \mathcal{S}(\R)$, and $1 \leq p,q \leq \infty$. Then the modulation space $M^{p,q}(\R, A_{\alpha, \beta})$  consists of all tempered distributions $f \in \mathcal{S'}(\R)$ such that $V_g f \in L^{p,q}(\R^2, A_{\alpha, \beta})$. The norm on $M^{p,q}(\R, A_{\alpha, \beta})$ is 
\begin{eqnarray*}
\Vert f \Vert_{M^{p,q}(\R, A_{\alpha, \beta})}
&=& \Vert V_g f \Vert_{L^{p,q}(\R^2, A_{\alpha, \beta})} \\
&=& \bigg( \int_{\R} \bigg( \int_{\R} |V_gf(x,w)|^p A_{\alpha, \beta}(x) dx \bigg)^{q/p} A_{\alpha, \beta}(w) dw \bigg)^{1/q} < \infty,
\end{eqnarray*}
with the usual adjustments if $p$ or $q$ is infinite. If $p=q$, then we write $M^p(\R, A_{\alpha, \beta})$ instead of $M^{p,p}(\R, A_{\alpha, \beta})$. Also,  we denote by $M^p(\R, \sigma_{\alpha, \beta})$ the modulation space corresponding to the measure $d |\sigma_{\alpha, \beta}|(x)$ and $M^p(\R)$ the modulation space corresponding to the Lebesgue measure $dx$.   
\end{Def}
The definition of $M^{p, q}(\R, A_{\alpha, \beta})$ is independent of the choice of $g$ in the sense that each different choice of $g$ defines an equivalent norm on $M^{p, q}(\R, A_{\alpha, \beta})$. Each modulation space is a Banach space. For $p=q=2$, we have that $M^2(\R, A_{\alpha, \beta}) =L^2(\R, A_{\alpha, \beta}).$ For other $p=q$, the space $M^p(\R, A_{\alpha, \beta})$ is not $L^p(\R, A_{\alpha, \beta})$. In fact for $p=q>2$, the space $M^p(\R, A_{\alpha, \beta})$ is a superset of $L^2(\R, A_{\alpha, \beta})$. We have the following inclusion 
\[ \mathcal{S}(\R) \subset M^1(\R, A_{\alpha, \beta}) \subset M^2(\R, A_{\alpha, \beta})=L^2(\R, A_{\alpha, \beta}) \subset M^\infty(\R, A_{\alpha, \beta}) \subset \mathcal{S'}(\R). \]
In particular, we have $M^p(\R, A_{\alpha, \beta}) \hookrightarrow L^p(\R, A_{\alpha, \beta})$ for $1 \leq p \leq 2$, and  $L^p(\R, A_{\alpha, \beta}) \hookrightarrow M^p(\R, A_{\alpha, \beta})$ for $2 \leq p \leq \infty$. Furthermore, the dual of a modulation space is also a modulation space, if $p < \infty$, $q < \infty$, $(M^{p, q}(\R, A_{\alpha, \beta}))^{'} =M^{p', q'}(\R, A_{\alpha, \beta})$, where $p', \; q'$ denote the dual exponents of $p$ and $q$, respectively. We refer to Gr\"ochenig's book \cite{gro01} for further properties and uses of modulation spaces. 

\section{Cowling--Price's theorem for the Opdam--Cherednik transform}\label{sec4}
We begin this section with the following lemma which we need for the proof of the next result.
\begin{lemma}\label{lem3}
If $f(t)=1$ and $g(t)=e^{-\pi t^2}$, then
\[ V_g f (x, w) =e^{- 2 \pi i w x} \; e^{-\pi w^2}. \]
Also for $p\in [1, \infty)$ and $\rho, \sigma >0$, we have
\[ \|f\|_{M^p([\rho \sigma, \; \rho (\sigma+1)])} \leq  \rho^{\frac{2}{p}}.  \]
\end{lemma}
\begin{proof}
\begin{eqnarray*}
V_g f (x, w) 
&=& \int_{\R} e^{-\pi (t-x)^2} e^{- 2 \pi i w t} dt \\
&=& \int_{\R} e^{-\pi s^2} e^{- 2 \pi i w (s+x)} ds \\
&=& e^{- 2 \pi i w x} \int_{\R} e^{-\pi s^2} e^{- 2 \pi i w s} ds \\
&=& e^{- 2 \pi i w x} \; e^{-\pi w^2}. 
\end{eqnarray*}
Also, we have
\begin{eqnarray*}
\|f\|_{M^p([\rho \sigma, \; \rho (\sigma+1)])}
&=& \| V_g f \|_{L^p([\rho \sigma, \; \rho (\sigma+1)]\times [\rho \sigma, \; \rho (\sigma+1)])} \\
&=& \left( \int_{\rho \sigma}^{\rho (\sigma+1)} \int_{\rho \sigma}^{\rho (\sigma+1)} e^{- \pi p w^2} dx dw \right)^{\frac{1}{p}}\\
& \leq &  \left( \int_{\rho \sigma}^{\rho (\sigma+1)} \int_{\rho \sigma}^{\rho (\sigma+1)} dx dw \right)^{\frac{1}{p}}\\
&=&  \rho^{\frac{2}{p}}. 
\end{eqnarray*}
\end{proof}

We obtain the following lemma of Phragm{\'e}n--Lindl{\"o}f type using the same technique as in \cite{cow83}. This lemma plays a crucial role in the proof of the next theorem, which is an $M^p$ -- $M^q$-version of Cowling--Price's theorem for the Opdam--Cherednik transform. An $L^p$-version of the following lemma proved in \cite{cow83} but here we prove the lemma for the modulation space $M^p(\R, \sigma_{\alpha, \beta})$ and obtain a different estimate.

\begin{lemma}\label{lem4}
Suppose that $g$ is analytic in the region $ Q = \{re^{i \theta} : r>0, \;0 < \theta < \frac{\pi}{2}\}$ and continuous on the closure $\bar{Q}$ of $Q$. Suppose also that for $p\in [1, \infty)$ and constants $A, a > 0$,
\[|g(x + iy)| \leq A \; e^{ax^2} \quad \text{for} \;\; x + iy \in \bar{Q} , \]
and
\[  \| g_{|\R} \|_{M^p(\R, \sigma_{\alpha, \beta})} \leq A. \]
Then
\[ \int_\sigma^{\sigma+1} |g(\rho e^{i \psi} )| \; d\rho \leq  A \; \max \left\{e^a , (\sigma + 1)^{\frac{2}{p}-1} \right\} \]
for $\psi \in [0, \frac{\pi}{2}]$ and $\sigma \in \R^+$. 
\end{lemma}
\begin{proof}
Using the definition of $M^p(\R, \sigma_{\alpha, \beta})$ and the fact that there is a constant $k_1 > 0$ such that $|C_{\alpha, \beta} (\lambda)|^{-2} \geq k_1 |\lambda|^{2 \alpha +1}$ for all $\lambda \in \R$  with $|\lambda| \geq 1$ (see \cite{tri97}, page 157), we get
\begin{eqnarray*}
&& \| g_{|\R} \|^p_{M^p(\R, \sigma_{\alpha, \beta})} = \| V_h g \|^p_{L^p(\R^2, \sigma_{\alpha, \beta})} \\
&& \geq  \int_{|\mu| \geq 1} \int_{|\lambda| \geq 1} |V_h g(\lambda, \mu)|^p \; d |\sigma_{\alpha, \beta}|(\lambda) \; d |\sigma_{\alpha, \beta}|(\mu) \\
&& = \int_{|\mu| \geq 1} \int_{|\lambda| \geq 1} |V_h g(\lambda, \mu)|^p \;\left|1-\frac{\rho}{i\lambda}\right| \frac{d\lambda}{8 \pi |C_{\alpha, \beta} (\lambda)|^2}  \; \left|1-\frac{\rho}{i\mu}\right| \frac{d\mu}{8 \pi |C_{\alpha, \beta} (\mu)|^2} \\
&& \geq \frac{1}{64\pi^2} \int_{|\mu| \geq 1} \int_{|\lambda| \geq 1} |V_h g(\lambda, \mu)|^p \; \frac{d\lambda}{|C_{\alpha, \beta} (\lambda)|^2}  \;  \frac{d\mu}{|C_{\alpha, \beta} (\mu)|^2}\\
&& \geq \frac{k_1^2}{64\pi^2} \int_{|\mu| \geq 1} \int_{|\lambda| \geq 1} |V_h g(\lambda, \mu)|^p  |\lambda|^{2 \alpha +1} d\lambda \; |\mu|^{2 \alpha +1} d\mu \\
&& \geq \frac{k_1^2}{64\pi^2} \int_{|\mu| \geq 1} \int_{|\lambda| \geq 1} |V_h g(\lambda, \mu)|^p \;  d\lambda d\mu,
\end{eqnarray*}
where $h \in \mathcal{S}(\R)$. This shows that $\| g_{|\R} \|_{M^p(\R)} \leq A$. Now, we define a function $f$ on $\bar{Q}$ by $ f(z)=g(z) \exp (i \varepsilon e^{i \varepsilon} z^{\frac{(\pi - 2 \varepsilon)}{\theta}}+i a \cot (\theta) z^2/2 )$ for $\theta \in (0, \pi/2)$ and $\varepsilon \in (0, \pi/2-\theta)$. We apply the same method as in \cite{cow83} to the function $f$ to get the estimate. By using H{\"o}lder's inequality for $M^p$ and Lemma \ref{lem3}, for $\rho > (\sigma+1)^{-1}$ we obtain 
\begin{eqnarray*}
\int_\sigma^{\sigma+1} |f(\rho \tau)| d \tau 
= \frac{1}{\rho} \int_{\rho \sigma}^{\rho (\sigma+1)} |f( \tau)| d\tau 
&\leq & \frac{1}{\rho}\; \|f\|_{M^p([\rho \sigma, \; \rho (\sigma+1)])} \|1 \|_{M^q([\rho \sigma, \; \rho (\sigma+1)])} \\
& \leq & \rho^{\frac{2}{q}-1} \|g\|_{M^p([\rho \sigma, \; \rho (\sigma+1)])} \\
& \leq & A \; (\sigma+1)^{\frac{2}{p}-1}.
\end{eqnarray*}
Now the rest of the proof follows similarly as in \cite{cow83}.
\end{proof} 

\begin{theorem}\label{th2}
Let $1 \leq p, q \leq \infty$ with at least one of them finite.
Suppose that $f$ is a measurable function on $\R$ such that 
\begin{eqnarray}\label{eq5}
e^{a x^2} f \in M^p(\R, A_{\alpha, \beta}) \quad \text{and} \quad  e^{b \lambda^2} \H f \in M^q(\R, \sigma_{\alpha, \beta}),
\end{eqnarray}
for some constants $a, b > 0$. Then the following conclusions hold:
\begin{enumerate}
\item[(i)] If $ab \geq \frac{1}{4}$, then $f = 0$ almost
everywhere.

\item[(ii)] If $ab < \frac{1}{4}$, then for all $t \in (b, \frac{1}{4a})$, the functions $f = E^{\alpha, \beta}_t$ satisfy the relations $(\ref{eq5})$.
\end{enumerate}
\end{theorem} 

\begin{proof}
We divide the proof in several steps.

Step 1: Assume that $ab > \frac{1}{4}$. The function
\[ \H f (\lambda)=\int_{\R} f(x)\; G^{\alpha, \beta}_\lambda(-x)\; A_{\alpha, \beta} (x) dx, \quad \text{for any } \lambda \in \C, \]
is well defined, entire on $\C$, and satisfies the condition
\begin{eqnarray}\label{eq6}
|\H f (\lambda)| & \leq &  \int_{\R} |f(x)|\; |G^{\alpha, \beta}_\lambda(-x)|\; A_{\alpha, \beta} (x) dx \nonumber \\
& \leq & C \int_{\R} |f(x)|\; e^{|\mathrm{Im}(\lambda)| |x|}\; A_{\alpha, \beta} (x) dx, \quad  \; \text{by  } (\ref{eq1}), \nonumber \\
& = & C \; e^{\frac{|\mathrm{Im}(\lambda)|^2}{4a}}
\int_{\R}  e^{a x^2} |f(x)|\;  e^{- a \left(x - \frac{|\mathrm{Im}(\lambda)|}{2a}\right)^2 }\; A_{\alpha, \beta} (x) dx, \;\; \text{so by H\"older's inequality}, \nonumber \\
& \leq & C \; e^{ \frac{(\mathrm{Im}(\lambda))^2}{4a}} \; \Big\Vert e^{a x^2} f \Big\Vert_{M^p(\R, A_{\alpha, \beta})} \; \Big\Vert e^{- a \left(x - \frac{|\mathrm{Im}(\lambda)|}{2a}
\right)^2} \Big\Vert_{M^{p'}(\R, A_{\alpha, \beta})}, 
\end{eqnarray}
where $p'$ is the conjugate exponent of $p$. We consider the function $g$ defined on $\C$ by
\begin{equation}\label{eq7}
g(\lambda)=e^{\frac{\lambda^2}{4a}} \; \H f (\lambda) .
\end{equation}
Then $g$ is an entire function on $\C$ and using (\ref{eq6}), we obtain that there exists a constant $A$ such that
\begin{equation}\label{eq8}
|g(\lambda)| \leq A \;  e^{\frac{(\mathrm{Re}(\lambda))^2}{4a}}, \quad \text{for all } \lambda \in \C.
\end{equation}
In the following we consider two cases.

(i) Let $q<\infty$. Using the inequality $ab > \frac{1}{4}$ and the hypothesis (\ref{eq5}), we obtain
\begin{equation}\label{eq9}
\| g_{|\R} \|_{M^q(\R, \sigma_{\alpha, \beta})}=\left\| e^{b \lambda^2} \; \H f \;  e^{\left(\frac{1}{4a} - b\right)\lambda^2} \right\|_{M^q(\R, \sigma_{\alpha, \beta})} \leq \left\| e^{b \lambda^2} \; \H f  \right\|_{M^q(\R, \sigma_{\alpha, \beta})} \leq A.
\end{equation}
By applying the Lemma \ref{lem4} to the functions $g(\lambda), \; g(-\lambda), \; \overline{g(\overline{\lambda})}$  and  $\overline{g(-\overline{\lambda})}$, we get that for $\psi \in [0, 2\pi]$ and large $\sigma$
\[ \int_\sigma^{\sigma+1} |g(\rho e^{i \psi} )| \; d\rho \leq  B (\sigma + 1)^{\frac{2}{q}-1},  \]
for some constant $B$. Now by Cauchy's integral formula,
\[ |g^{(n)} (0)| \leq n! (2 \pi)^{-1} \int_0^{2 \pi} |g(\rho e^{i \psi} )| \; \rho^{-n} d\psi.  \]
Consequently, for large $\sigma$,
\begin{eqnarray}\label{eq10}
|g^{(n)} (0)| & \leq & n! (2 \pi)^{-1} \int_0^{2 \pi} \left( \int_\sigma^{\sigma+1}  |g(\rho e^{i \psi} )| \; \rho^{-n}  d\rho \right) d\psi \\
& \leq & B n! \; \sigma^{-n}  (\sigma + 1)^{\frac{2}{q}-1}. \nonumber
\end{eqnarray}
Let $\sigma \to \infty$. If $q \geq 2$, then $(\sigma + 1)^{\frac{2}{q}-1} \leq B_1$, for some constant $B_1$ and $g^{(n)} (0)=0$ for $n \geq 1$. So $g(\lambda)= D$, for some constant $D$. From (\ref{eq9}), $g(\lambda) = 0$ for all $\lambda \in \C$. Also, if $q < 2$, then $g^{(n)} (0)=0$ for $n \geq 2$. So $g(\lambda)= C \lambda+ D$, for some constants $C$ and $D$. From (\ref{eq8}) and (\ref{eq9}), $g(\lambda) = 0$ for all $\lambda \in \C$. Thus $\H f (\lambda)=0$ for all $\lambda \in \R$, and then  by (\ref{eq03}), we have $f=0$ almost everywhere on $\R$. 
  
(ii) Let $q = \infty$. As $ab > \frac{1}{4}$, then from (\ref{eq5}) we obtain
\begin{equation}\label{eq11}
\| g_{|\R} \|_{M^\infty(\R, \sigma_{\alpha, \beta})} \leq \left\| e^{b \lambda^2} \; \H f  \right\|_{M^\infty(\R, \sigma_{\alpha, \beta})} < \infty. 
\end{equation}
If we consider $q=\infty$, then the estimate obtained in Lemma \ref{lem4} can be refined so that $\max \{e^a , (\sigma + 1)^{\frac{2}{q}-1} \}$ is replaced by $1$ (see \cite{cow83}). From (\ref{eq10}), we have
\[ |g^{(n)} (0)| \leq   A \; n! \; \sigma^{-n}.  \]
Then $g^{(n)} (0)=0$ for $n \geq 1$. So $g(\lambda) = C$ for all $\lambda \in \C$, for some constant $C$. Therefore $e^{b \lambda^2} \H f (\lambda)=C \; e^{(b-\frac{1}{4a})\lambda^2}$ for all $\lambda \in \R$. Since $ab > \frac{1}{4}$, this function satisfies the relation (\ref{eq11}) implies that $C=0$ thus,  $f = 0$ almost everywhere on $\R$.

Step 2: Assume that $ab = \frac{1}{4}$.

(a) If $q < \infty$, with the same proof as for the point (i) of the first step, we obtain $f = 0$ almost everywhere on $\R$.

(b) Let $q = \infty$ and $1 \leq p < \infty$. We have $\| g_{|\R} \|_{M^\infty(\R, \sigma_{\alpha, \beta})} < \infty$. Then by the point (ii) of the first step, the relation (\ref{eq7}), and the property (\ref{eq05}) of the heat kernel $E^{\alpha, \beta}_{\frac{1}{4a}}$, we deduce that
\begin{equation}\label{eq12}
\H f (\lambda)=C \; e^{-\frac{\lambda^2}{4a}}=C \; \H ( E^{\alpha, \beta}_{\frac{1}{4a}} ) (\lambda), \quad \text{for all } \lambda \in \R, 
\end{equation}
for some constant $C$. Thus from the injectivity of the transform $\H$, we obtain
\begin{equation}\label{eq14}
f(x)= C \;  E^{\alpha, \beta}_{\frac{1}{4a}} (x), \quad \text{a.e. } x \in \R.
\end{equation} 
By using the relations (\ref{eq06}) and (\ref{eq14}), we get
\[ \frac{2 C e^{\frac{\mu_1}{4a}} a^{\alpha+1}}{\Gamma(\alpha+1) \sqrt{B_{\alpha, \beta}(x)}} \leq e^{a x^2} f(x), \quad \text{for all } x \in \R. \]
From the properties of the functions $A_{\alpha, \beta}$ and $B_{\alpha, \beta}$, we obtain that for finite $p$ 
\[ \left\| \frac{1}{\sqrt{B_{\alpha, \beta} (x)}} \right\|_{M^p(\R, \;A_{\alpha, \beta})} =\infty. \]
On the other hand, from (\ref{eq5}) we have $\| e^{a x^2} f \|_{M^p(\R, \;A_{\alpha, \beta})} < \infty$, this is impossible unless $C = 0$. Then we obtain from (\ref{eq14}) that $f = 0$ almost everywhere on $\R$.
    
Step 3: Assume that $ab < \frac{1}{4}$. Let $t \in (b, \frac{1}{4a})$ and $f = E^{\alpha, \beta}_t$.  From the relation (\ref{eq06}), we get 
\[ K_1 e^{-\left(\frac{1}{4t}-a\right)x^2} \leq e^{ax^2} f(x) \leq K_2 e^{-\left(\frac{1}{4t}-a\right)x^2}, \quad \text{for all }  x \in \R, \] 
for some constants $K_1, \;K_2 > 0$. As $t < \frac{1}{4a}$, we deduce that  $e^{ax^2} f \in  {M^p(\R, A_{\alpha, \beta})} $.  Using the relation (\ref{eq04}), we get
\[ e^{b \lambda^2} \H f (\lambda)=e^{-(t-b)\lambda^2}, \quad \text{for all }  \lambda \in \R. \]
The condition $t > b$ and the inequality $|C_{\alpha, \beta} (\lambda)|^{-2} \leq k_2 |\lambda|^{2 \alpha +1}$ at infinity (see \cite{tri97}, page 157) imply that $e^{b \lambda^2} \H f \in M^q(\R, \sigma_{\alpha, \beta})$. This completes the proof of the theorem.
\end{proof}

\section{Hardy's theorem for the Opdam--Cherednik transform}\label{sec5}

In this section, we determine the functions $f$ satisfying the relations (\ref{eq5})  in the special case $p = q = \infty$. The result we obtain, is an analogue of the classical Hardy's theorem for the Opdam--Cherednik transform.

\begin{theorem}
Let $f$ be a measurable function on $\R$ such that
\begin{equation}\label{eq15}
e^{a x^2} f \in M^\infty(\R, A_{\alpha, \beta}) \quad \text{and} \quad  e^{b \lambda^2} \H f \in M^\infty(\R, \sigma_{\alpha, \beta}),
\end{equation}
for some constants $a, b > 0$. Then
\begin{enumerate}
\item[(i)] If $ab > \frac{1}{4}$, we have $f = 0$ almost
everywhere.

\item[(ii)] If $ab = \frac{1}{4}$, the function $f$ is of the form $f = C  E^{\alpha, \beta}_{\frac{1}{4a}} $, for some real constant $C$.

\item[(iii)] If $ab < \frac{1}{4}$, there are infinitely many nonzero functions $f$ satisfying the conditions $(\ref{eq15})$.
\end{enumerate}
\end{theorem}
\begin{proof}
(i) If $ab > \frac{1}{4}$, the point (ii) of the first step of the proof of Theorem \ref{th2} gives the result.

(ii) If $ab = \frac{1}{4}$ and $\| e^{b \lambda^2} \H f  \|_{M^\infty(\R, \; \sigma_{\alpha, \beta})} < \infty$, then from Step 2 (b) of the proof of Theorem \ref{th2} and the relation (\ref{eq14}), we have $f = C  E^{\alpha, \beta}_{\frac{1}{4a}} $, for some real constant $C$. As $B_{\alpha, \beta} (x) \geq 1$, from relations (\ref{eq06}) and (\ref{eq14}), we get that
\[ e^{a x^2} f(x) \leq \frac{2 C e^{\frac{\mu_2}{4a}} a^{\alpha+1}}{\Gamma(\alpha+1) \sqrt{B_{\alpha, \beta}(x)}}, \quad \text{for all } x \in \R. \]
On the other hand, from (\ref{eq15}) we have $\| e^{a x^2} f \|_{M^\infty(\R, \;A_{\alpha, \beta})} < \infty$, this is impossible unless $f = C  E^{\alpha, \beta}_{\frac{1}{4a}} $. Thus, the result of point (ii) is proved.

(iii) If $ab < \frac{1}{4}$, the functions $f = E^{\alpha, \beta}_t$, $t \in (b, \frac{1}{4a})$, satisfy the conditions $(\ref{eq15})$. This completes the proof of the theorem.
\end{proof}  

\section{Morgan's theorem for the Opdam--Cherednik transform}\label{sec6}

The aim of this section is to prove an $M^p$ -- $M^q$-version of Morgan's theorem for the Opdam--Cherednik transform. Before we prove the main result of this section, we first need the following lemma. 

\begin{lemma}[\cite{far03}, Lemma 2.3]\label{lem1}
Suppose that $\rho \in (1, 2)$, $q \in [1, \infty]$, $\eta > 0$, $M > 0$ and $ B > \eta \sin \frac{\pi}{2} (\rho-1)$. If $g$ is an entire function on $\C$ satisfying the conditions
\begin{enumerate}
\item[(i)] $|g(x + iy)| \leq M e^{\eta |y|^\rho},\;$ for any $x,y \in \R$, 
\item[(ii)] $ e^{B |x|^\rho} g_{|\R}  \in L^q(\R) $, 
\end{enumerate}
then $g = 0$.
\end{lemma}

Using the above lemma, in the following, we obtain a version of the Phragm{\'e}n--Lindl{\"o}f type result for the modulation spaces.
\begin{lemma}\label{lem2}
Suppose that $\rho \in (1, 2)$, $q \in [1, \infty)$, $\eta > 0$, $M > 0$ and $B > \eta \sin \frac{\pi}{2} (\rho-1)$. If $g$ is an entire function on $\C$ satisfying the conditions
\begin{enumerate}
\item[(i)] $|g(x + iy)| \leq M e^{\eta |y|^\rho}$, for any
$x, \;y \in \R$,
\item[(ii)] $ e^{B |x|^\rho} g_{|\R}  \in M^q(\R) $,
\end{enumerate}
then $g = 0$.
\end{lemma}
\begin{proof}
Let $R > 0$ be such that 
\[B > \eta ((R+1)/R)^\rho \sin \frac{\pi}{2} (\rho-1).\] 
Consider the entire function on $\C$ defined by 
\[ F(z)=\int_R^{R+1} g(tz) dt.\]
Then for any $n \in \N$, the derivatives of $F$ satisfy the condition 
\[ F^{(n)}(0) = \left[\left( (R + 1)^{n+1} - R^{n+1}\right) /(n+1)\right] g^{(n)}(0).  \]
Therefore, $g =0$ if and only if $F =0$. By assumption (i), we have
\begin{equation}\label{eq01}
|F(x+iy)| \leq  M \;e^{{(R+1)}^{\rho} \eta |y|^{\rho}},\mathrm{\;for \;any\;} x, \;y \in \R.
\end{equation}
Let $x \in \R \setminus \{0\}$, the change of variable $u = xt$ gives
\[ F(x)= \frac{1}{x} \int_{Rx}^{(R+1)x} g(u) du,\]
so 
\begin{eqnarray*}
|F(x)| \leq \frac{1}{|x|} \int_{Rx}^{(R+1)x} |g(u)| du
\leq \frac{1}{|x|} \int_{Rx}^{(R+1)x} |g(u)| \;e^{B |u|^\rho} e^{- R^\rho B |x|^\rho} du.
\end{eqnarray*}
By the H\"older's inequality, we get
\[|F(x)| \leq\frac{1}{|x|} \;\Vert e_B g \Vert_{M^q(\R)}
\; \Vert 1 \Vert_{M^{q'}([Rx, (R+1)x])} \; e^{- R^\rho B |x|^\rho}, \]
where $e_B(u)=e^{B |u|^\rho}$ and $q'$ is the conjugate exponent of $q$. Since 
$$\Vert 1 \Vert_{M^{q'}([Rx, (R+1)x])} \leq C |x|^{1/{q'}}$$ 
for some constant $C>0$, we have
\[|F(x)| \leq \frac{C}{|x|^{1/q}} \;\Vert e_B g \Vert_{M^q(\R)} \; e^{- R^\rho B |x|^\rho}. \]
Since $F$ is continuous on $\R$, using assumption (ii), we obtain 
\begin{equation}\label{eq02}
e^{R^\rho B |x|^\rho} F_{|\R} \in L^\infty(\R).
\end{equation}
Using the inequalities (\ref{eq01}) and (\ref{eq02}), and applying Lemma \ref{lem1} for $q=\infty$ to $F$, we see that $F=0$, thus $g=0$. This completes the proof.
\end{proof}

\begin{theorem}
Let $ p \in [1, \infty]$, $ q \in [1, \infty)$, $a>0$, $b>0$, and let $\alpha, \beta$ be positive real numbers satisfying $\alpha >2$ and $1/\alpha+1/\beta=1$. Suppose that $f$ is a measurable function on $\R$ such that  
\[ e^{a|x|^\alpha} f \in M^p(\R, A_{\alpha, \beta}) \quad \text{and} \quad  e^{b|\lambda|^\beta} \H f \in M^q(\R, \sigma_{\alpha, \beta}). \]
If 
\[ (a \alpha)^{1/\alpha} (b \beta)^{1/\beta} > \left( \sin \left( \frac{\pi}{2}(\beta -1 ) \right) \right)^{1/\beta},  \]
then $f = 0$. 
\end{theorem}

\begin{proof}
Let $f$ be a measurable function on $\R$ such that
\begin{equation}\label{eq2}
e^{a|x|^\alpha} f \in M^p(\R, A_{\alpha, \beta})
\end{equation}
and
\begin{equation}\label{eq3}
e^{b|\lambda|^\beta} \H f \in M^q(\R, \sigma_{\alpha, \beta}).
\end{equation}
We use conditions (\ref{eq2}) and (\ref{eq3}) to prove that the Opdam--Cherednik transform of $f$ satisfies the conditions (i) and (ii) of Lemma \ref{lem2}, and hence we deduce that $f=0$ almost everywhere.

The function
\[ \H f (\lambda)=\int_{\R} f(x)\; G^{\alpha, \beta}_\lambda(-x)\; A_{\alpha, \beta} (x) dx \] 
is well defined, entire on $\C$, and satisfies the condition
\begin{eqnarray*}
|\H f (\lambda)| & \leq &  \int_{\R} |f(x)|\; |G^{\alpha, \beta}_\lambda(-x)|\; A_{\alpha, \beta} (x) dx \\
& \leq & C \int_{\R} |f(x)|\; e^{|\mathrm{Im}(\lambda)| |x|}\; A_{\alpha, \beta} (x) dx, \quad \text{for any } \lambda \in \C, \; \text{by  } (\ref{eq1}), \\
& \leq & C \; \Big\Vert e^{a|x|^\alpha} f \Big\Vert_{M^p(\R, A_{\alpha, \beta})} \; \Big\Vert e^{-a|x|^\alpha} e^{|\mathrm{Im}(\lambda)| |x|} \Big\Vert_{M^{p'}(\R, A_{\alpha, \beta})}, \quad \text{by H\"older's inequality}, \\
& \leq & C_1 \; \Big\Vert e^{|\mathrm{Im}(\lambda)| |x|-a|x|^\alpha} \Big\Vert_{M^{p'}(\R, A_{\alpha, \beta})}, \quad \text{by } (\ref{eq2}),
\end{eqnarray*}
where $C_1$ is a constant and $p'$ is the conjugate exponent of $p$. 

Let
\[C \in I= \left( (b \beta)^{-1/\beta} \left( \sin \left( \frac{\pi}{2}(\beta -1 ) \right) \right)^{1/\beta}, \; (a \alpha)^{1/\alpha}  \right) . \]
Applying the convex inequality
\[ |ty| \leq \left(\frac{1}{\alpha}\right)|t|^\alpha + \left(\frac{1}{\beta}\right)|y|^\beta \]
to the positive numbers $C|t|$ and $|y|/C$, we obtain
\[ |ty| \leq \frac{C^\alpha}{\alpha} \;|t|^\alpha + \frac{1}{\beta C^\beta} \;|y|^\beta, \]
and thus
\[ \Big\Vert e^{|\mathrm{Im}(\lambda)| |x|-a|x|^\alpha} \Big\Vert_{M^{p'}(\R, A_{\alpha, \beta})} \leq e^{{|\mathrm{Im}(\lambda)|^\beta}/{\beta C^\beta}} \; \Big\Vert e^{-\left(a- {C^\alpha}/{\alpha} \right)|x|^\alpha} \Big\Vert_{M^{p'}(\R, A_{\alpha, \beta})}. \]
Since $C \in I$ , it follows that $a > C^\alpha / \alpha$, and thus   
\[\Big\Vert e^{-\left(a- {C^\alpha}/{\alpha} \right)|x|^\alpha} \Big\Vert_{M^{p'}(\R, A_{\alpha, \beta})} < \infty .  \]
Therefore, 
\[\Big\Vert e^{|\mathrm{Im}(\lambda)| |x|-a|x|^\alpha} \Big\Vert_{M^{p'}(\R, A_{\alpha, \beta})}   < \infty. \]
Moreover,
\begin{equation}\label{eq4}
|\H f (\lambda)|  \leq \text{Const.}\; e^{{|\mathrm{Im}(\lambda)|^\beta}/{\beta C^\beta}}  \quad \text{for any } \lambda \in \C.
\end{equation}
Condition (\ref{eq3}) and inequality (\ref{eq4}) imply that the function $g(z)=\H f (z)$ satisfies the assumptions (i) and (ii) of Lemma \ref{lem2} with $\rho = \beta$, $\eta=1/(\beta C^\beta)$, and $B=b$. The condition $C \in I$ implies the inequality 
\[ b > \frac{1}{\beta C^\beta}  \sin \left( \frac{\pi}{2}(\beta -1 ) \right),  \]
which gives $\H f=0$ by Lemma \ref{lem2}, then $f=0$ by (\ref{eq03}). 
\end{proof}

\section*{Acknowledgments}
The author is deeply indebted to Prof. S. Thangavelu for several fruitful discussions and generous comments. The author is grateful to the University Grants Commission, India for providing the Dr. D. S. Kothari Post Doctoral Fellowship (Award No.- F.4-2/2006 (BSR)/MA/18-19/0032). The author also wishes to thank the anonymous referee for valuable comments and suggestions which helped to improve the quality of the paper.

\section*{Disclosure statement}
No potential conflict of interest was reported by the author.

\end{document}